\documentclass[12pt,a4paper]{amsart}
\usepackage{amsmath,amssymb,amsthm,latexsym}
\usepackage[utf8]{inputenc}
\usepackage{tikz-cd}
\usepackage[colorlinks=true]{hyperref}
\usepackage[a4paper,portrait]{geometry}
\usepackage{todonotes}
\usepackage{graphicx}
\usepackage{color}
\usepackage{subfigure}

\numberwithin{equation}{section}

\newcommand{\Z}{\mathbb{Z}}
\newcommand{\Q}{\mathbb{Q}}

\newcommand{\R}{\mathbb{R}}
\newcommand{\C}{\mathbb{C}}

\DeclareMathOperator{\Gal}{Gal}

\DeclareMathOperator{\im}{Im}

\DeclareMathOperator{\re}{Re}

\DeclareMathOperator{\SL}{SL}
\DeclareMathOperator{\sgn}{sgn}

\newtheorem{thm}{Theorem}[section]
\newtheorem*{thm*}{Theorem}
\newtheorem{lem}[thm]{Lemma}
\newtheorem{pro}[thm]{Proposition}
\newtheorem{cor}[thm]{Corollary}
\newtheorem*{cor*}{Corollary}

\newtheorem{algo}[thm]{Algorithm}
\theoremstyle{definition}

\theoremstyle{remark}

\begin{document}

\author[Zhengyu Tao]{Zhengyu Tao}


\address{Department of Mathematics, Nanjing University, Nanjing 210093, People's Republic of China}

\email{taozhy@smail.nju.edu.cn}

\author[Xuejun Guo]{Xuejun Guo$^{\ast}$}

\address{Department of Mathematics, Nanjing University, Nanjing 210093, People's Republic of China}

\email{guoxj@nju.edu.cn}

\title[CM points, class numbers, and Mahler measures]{CM points, class numbers, and the Mahler measures of $x^3+y^3+1-kxy$}



\thanks{The authors are supported by NSFC 11971226 and NSFC 12231009}

\thanks{$^{\ast}$Corresponding author}

\begin{abstract}
We study the Mahler measures of the polynomial family $Q_k(x,y) = x^3+y^3+1-kxy$ using the method previously developed by the authors. An algorithm is implemented to search for CM points with class numbers $\leqslant 3$, we employ these points to derive interesting formulas that link the Mahler measures of $Q_k(x,y)$ to $L$-values of modular forms. As by-products, some conjectural identities of Samart are confirmed, one of them involves the modified Mahler measure $\tilde{n}(k)$ introduced by Samart recently. For $k=\sqrt[3]{729\pm405\sqrt{3}}$, we also prove an equality that expresses a $2\times 2$ determinant with entries the Mahler measures of $Q_k(x,y)$ as some multiple of the $L$-value of two isogenous elliptic curves over $\Q(\sqrt{3})$. 
\end{abstract}

\maketitle

\section{Introduction}

For any non-zero Laurent polynomial $P\in\C[x_1^{\pm1},\cdots,x_n^{\pm1}]$, the (logarithmic) Mahler measure of $P$ is defined by
\[m(P)=\frac{1}{(2\pi)^n}\int_{0}^{2\pi}\cdots\int_{0}^{2\pi}\log|P(e^{i\theta_1},\cdots,e^{i\theta_n})|d\theta_1\cdots d\theta_n.\]

Initiated by the insights of Deninger and Boyd, the relation between multivariate Mahler measures and special values of $L$-functions has attracted a significant amount of research. In \cite{Den97}, Deninger conjectured that 
\begin{equation}\label{Deningerconj}
	m\Bigl(x+\frac{1}{x}+y+\frac{1}{y}+1\Bigr) \overset{?}{=} L'(E,0),
\end{equation}
where $E$ is the conductor 15 elliptic curve defined by the projective closure of $x+\frac{1}{x}+y+\frac{1}{y}+1=0$. Later,  based on numerical experiments, Boyd \cite{Boy98} made similar conjectures of the form
\begin{equation}\label{Boydconj}
	m(P_k)\overset{?}{=}r_k L'(E_k,0)
\end{equation}
for many $k\in\Z-\{0,\pm4\}$, where $P_k(x,y)=x+\frac{1}{x}+y+\frac{1}{y}+k, r_k\in\Q$ and $E_k$ is the elliptic curve associated to $P_k(x,y)=0$. He also formulated analogous conjectures for many other families, among which is the family
\begin{equation}
	Q_k(x,y)=x^3+y^3+1-kxy.
\end{equation}
Note that $Q_k(x,y)=0$ is the Hesse pencil of elliptic curves with Weierstrass model
\begin{equation}\label{EllipticcurveCk}
	C_k: Y^2=X^3-27k^6X^2+216k^9(k^3-27)X-432k^{12}(k^3-27)^2.
\end{equation}
The rational transformation that converts $C_k$ to the curve $Q_k(x,y)=0$ is
\begin{equation}\label{rationaltransformation}
	X=\frac{12k^4(k^3-27)x}{kx+3y+3},\quad Y=\frac{108k^6(k^3-27)(y-1)}{kx+3y+3}.
\end{equation}

During the period when Boyd's work appeared, motivated by mirror symmetry in physics, Rodriguez Villegas \cite{RV98} represented $m(P_k)$ and $m(Q_k)$ as Kronecker-Eisenstein series. This led to the proof of \eqref{Boydconj} in some cases when the elliptic curves $E_k$ and $C_k$ \emph{have} complex multiplication (usually abbreviated as CM). Specifically, Rodriguez Villegas proved \eqref{Boydconj} for the polynomials $P_{4\sqrt{2}},P_{2\sqrt{2}}$ and $Q_{-6}$.

Since the curve $E$ in \eqref{Deningerconj} has \emph{no} CM, Rodriguez Villegas' method doesn't work in this case. The question mark in equation \eqref{Deningerconj} was finally removed by Rogers and Zudilin \cite{RZ14} nearly twenty years after Deninger made his conjecture. Their approach, now known as the Rogers-Zudilin method \cite[Chapter 9]{BZ20}, can be successfully applied to prove a number of non-CM cases of \eqref{Boydconj}. However, the use of Rogers-Zudilin method relies heavily on the modular unit parametrization of elliptic curves and Brunault \cite{Bru15} proved that there are only finitely many elliptic curves over $\Q$ that can be parametrized
by modular units. Interested readers can refer to the tables in \cite{Sam21, Sam23}, where the proven cases of \eqref{Boydconj} related to $m(P_k)$ and $m(Q_k)$ are listed (whether CM or non-CM).

Although much of the current literature focuses on the study of non-CM cases, we believe that there are still some veins to be mined in the CM cases. In our previous work \cite{TGW22}, we proved that when $\tau$ is a CM point (i.e., imaginary quadratic numbers in the upper half plane $\mathcal{H}=\{\tau\in\C|\im(\tau)>0\}$), the degree of $k=k(\tau)$ as an algebraic number in Rodriguez Villegas' formula that expresses $m(P_k)$ as Kronecker-Eisenstein series can be bounded by the class number of the CM point $\tau$. This fact, together with a systematic search for CM points with class numbers $\leqslant 2$ enabled us to prove over twenty identities of the form
\[m(P_k)=\frac{r_ks_k}{\pi^2}L(f_k,2),\]
where $r_k\in\Q, s_k\in\{1,\sqrt{2},\sqrt{3},\sqrt{7}\}$ and $f_k$ are weight two cusp forms of levels $28, 48, 56, 64, 112,$  $128, 192, 256$ and $448$. Guided by Beilinson's
conjecture, we also proved $5$ identities connecting $L$-values of CM elliptic curves over real quadratic fields to $2\times 2$ determinants with $m(P_k)$ as entries. These identities extend the recent work \cite{GJLQ24} of Guo, Ji, Liu and Qin, in which they dealt with the case when $k=4\pm4\sqrt{2}$. As an example, we provide here one of our results:
\begin{equation}\label{mahlerasdet}
	\left|\det{\begin{pmatrix}m(P_{12+8\sqrt{2}})& m(P_{12-8\sqrt{2}})\\
		m(P_{12-8\sqrt{2}})&m(P_{12+8\sqrt{2}})\end{pmatrix}}\right|=\frac{1024}{\pi^4}L(E_{12\pm8\sqrt{2}},2).
\end{equation}

The present paper is devoted to treating the polynomial family $Q_k(x,y)$. Since a change of variables shows that $m(Q_k)$ only depends on $k^3$ \cite[\S 14]{RV98}, in the rest of this paper, we will use the following notation introduced by Samart \cite{Sam14}:
\[m_3(t):=3m(Q_{\sqrt[3]{t}})=3m(x^3+y^3+1-\sqrt[3]{t}xy).\]

The main difference between $P_k(x,y)$ and $Q_k(x,y)$ is that the family $P_k(x,y)$ is reciprocal while $Q_k(x,y)$ is nonreciprocal, where recall that a multivariable Laurent polynomial $P(x_1,\cdots, x_n)$ is \emph{reciprocal} if 
\[\frac{P(x_1,\cdots,x_n)}{P(1/x_1,\cdots,1/x_n)}=x_1^{b_1}\cdots x_n^{b_n}\]
for some $b_1,\cdots,b_n\in\Z$ and \emph{nonreciprocal} otherwise. Let $\mathcal{K}_Q$ (resp. $\mathcal{K}_P$) be the set of $k\in\C$ such that $Q_k(x,y)$ (resp. $P_k(x,y)$) vanishes on $\mathbb{T}^2=\{(x,y)\in\C^2\mid|x|=|y|=1\}$. As explained in \cite{Boy98, Sam23}, $\mathcal{K}_P\subset \R$ since $P_k$ is reciprocal, in fact $\mathcal{K}_P =[-4,4]$. However, for the nonreciprocal $Q_k(x,y)$, the set $\mathcal{K}_Q$ has non-empty interior: it is the compact region consisting of a hypocycloid with vertices at $3,3e^{\frac{2\pi i}{3}},3e^{\frac{4\pi i}{3}}$ (see \cite{Sam23} for the picture). We can now state Rodriguez Villegas' formula for $m_3(t)$:
\begin{thm}[Rodriguez Villegas {\cite[\S 14]{RV98}}]\label{Villegasthm}
	Let $\mathcal{F}'\subset\mathcal{H}$ be the fundamental domain of the congruence subgroup $\Gamma_0(3)$ formed by the geodesic triangle of vertices $i\infty,0,(1+i/\sqrt{3})/2$ and its reflection along the imaginary axis. For any $\tau\in \mathcal{F}'$, if $\sqrt[3]{t(\tau)}\in \C-\mathcal{K}^\circ_Q$, where $t(\tau)=27+\bigl(\eta(\tau)/\eta(3\tau)\bigr)^{12}$, then we have
	\begin{equation}\label{Villegasformula}
		m_3(t(\tau))=\frac{81\sqrt{3}\im(\tau)}{4\pi^2}\underset{m,n\in\Z}{{\sum}'}\frac{\chi_{-3}(n)(3m\re(\tau)+n)}{\left|3m\tau+n\right|^4},
	\end{equation}
	where $\chi_{-3}(\cdot)=\left(\frac{-3}{\cdot}\right)$ and $\underset{m,n\in\Z}{\sum'}$ means that $(m,n)=(0,0)$ is excluded from the summation.
\end{thm}

It is known that $t(\tau)$ is a Hauptmodul for $\Gamma_0(3)$, i.e., a generator of the function field of the modular curve $X_0(3)$. Let $\tau\in\mathcal{H}$ be a CM point. As mentioned earlier, this means that $\tau$ is an imaginary quadratic number in $\mathcal H$. Thus, there must exist three uniquely determined integers $a,b,c$ with $a>0, \gcd(a,b,c)=1$ such that $a\tau^2+b\tau+c=0$. In this paper, we will simply write $\tau$ as $[a,b,c]$. Recall that for any negative integer $D$ with $D\equiv 0$ or $1\mod 4$, the class number
\[h(D)=\#\{\text{primitive binary quadratic forms with discriminant } D\}/\sim,\]
where ``$\sim$'' is the equivalence relation that identifies equivalent quadratic forms as the same. As a slight abuse of notation, we define the \emph{class number} $h(\tau)$ of $\tau$ to be the class number of $b^2-4ac$, the discriminant of $\tau$. According to the theory of complex multiplication, $t(\tau)$ are algebraic numbers if $\tau$ takes CM points. Moreover, the algebraic degree of $t(\tau)$ can be bounded by $h(\tau)h(3\tau)$ (see Theorem \ref{boundoft}).

In \cite{Sam14}, Samart proved a number of formulas that express the Mahler measures of certain polynomials in two or three variables in terms of linear combinations of $L$-values of multiple modular forms. For the family $Q_k(x,y)$, he proved that
\begin{equation} \label{Samartresult}
	m_3\bigl(6-6\sqrt[3]{2}+18\sqrt[3]{4}\bigr)=\frac{3}{2}\bigl(L'(f_{108},0)+L'(f_{36},0)-3L'(f_{27},0)\bigr),
\end{equation}
where $f_{27}(\tau)=\eta(3\tau)^2\eta(9\tau)^2\in\mathcal{S}_2(\Gamma_0(27)), f_{36}(\tau)=\eta(6\tau)^4\in\mathcal{S}_2(\Gamma_0(36))$ and
\begin{equation*}
	\begin{split}
		f_{108}(\tau) & = \sum_{\substack{m,n\in\Z\\m\equiv\pm1,\pm2,n\equiv5\\(\mathrm{mod}\ 6)}}(4m+3n)q^{4m^2+6mn+3n^2}\\
		& =q+5q^7-7q^{13}-q^{19}-5q^{25}-4q^{31}-q^{37}+8q^{43}+\cdots
	\end{split}
\end{equation*}
is the unique normalized newform in $\mathcal{S}_2(\Gamma_0(108)).$ Since the elliptic curve $C_k$ has CM when $k=\sqrt[3]{6-6\sqrt[3]{2}+18\sqrt[3]{4}}$ (this can be easily checked by using the \textsf{SageMath} command \texttt{has\_cm()}), Theorem \ref{Villegasthm} should be able to resolve \eqref{Samartresult}. Indeed, Samart proved \eqref{Samartresult} by taking $\tau=\frac{i\sqrt{3}}{9}$ in \eqref{Villegasformula}. Since $\frac{i\sqrt{3}}{9}$ satisfies $27\tau^2+1$, we have $h\bigl(\frac{i\sqrt{3}}{9}\bigr)=h(-108)=3$. Based on some numerical values of the hypergeometric representation
of $m_3(t)$, Samart also discovered the following conjectural identities:
\begin{equation}\label{Samartconj1}
	m_3\bigl(17766+14094\sqrt[3]{2}+11178\sqrt[3]{4}\bigr)\overset{?}{=}\frac{3}{2}\bigl(L'(f_{108},0)+3L'(f_{36},0)+3L'(f_{27},0)\bigr),
\end{equation}
\begin{equation}\label{Samartconj2}
	m_3(\alpha\pm i\beta )\overset{?}{=}\frac{3}{2}\bigl(L'(f_{108},0)+3L'(f_{36},0)-6L'(f_{27},0)\bigr),
\end{equation}
\begin{equation}\label{Samartconj3}
	m_3\biggl(\frac{(7+\sqrt{5})^3}{4}\biggr)\overset{?}{=}\frac{1}{8}\bigl(9L'(f_{100},0)+38L'(f_{20},0)\bigr),
\end{equation}
\begin{equation}\label{Samartconj4}
	m_3\biggl(\frac{(7-\sqrt{5})^3}{4}\biggr)\overset{?}{=}\frac{1}{4}\bigl(9L'(f_{100},0)-38L'(f_{20},0)\bigr),
\end{equation}
where $\alpha=17766-7047\sqrt[3]{2}-5589\sqrt[3]{4}, \beta=243\sqrt{3}(29\sqrt[3]{2}-23\sqrt[3]{4}), f_{20}(\tau)=\eta(2\tau)^2\eta(10\tau)^2$ and $f_{100}(\tau)$ is a cusp form of level $100$. One can check that the elliptic curves $C_k$ related to \eqref{Samartconj1} and \eqref{Samartconj2} have CM, while those related to \eqref{Samartconj3} and \eqref{Samartconj4} have no CM. 
It is also worth noting that $17766+14094\sqrt[3]{2}+11178\sqrt[3]{4}$ and $\alpha\pm i\beta$ are the three roots of
\[T^3-53298T^2+1635876T-19683000=0.\]

In this paper, we apply the method developed in \cite{TGW22} to the family $Q_k(x,y)$ and obtain the following results for $k\in\C-\mathcal{K}_Q^\circ$.
\begin{thm}\label{mainresult}
	The following identities are true:
	\begin{equation}\label{iden1}
		m_3\bigl(-4320-1944\sqrt{5}\bigr)=\frac{405}{4\pi^2}L(F_{225},2),
	\end{equation}
	\begin{equation}\label{iden2}
		m_3\bigl(-163296-35640\sqrt{21}\bigr)=\frac{567}{4\pi^2}L(F_{441},2),
	\end{equation}
	\begin{equation}\label{iden3}
		m_3\bigl(729+405\sqrt{3}\bigr)=\frac{81}{\pi^2}L(F_{144},2),\quad m_3\bigl(729-405\sqrt{3}\bigr)=\frac{324}{\pi^2}L(\widetilde{F}_{144},2),
	\end{equation}
	\begin{equation}\label{iden4}
		m_3\bigl(17766\!+\!14094\sqrt[3]{2}\!+\!11178\sqrt[3]{4}\bigr)\!=\!\frac{243}{2\pi^2}L(F_{108},2),\quad m_3(\alpha\pm i\beta)\!=\!\frac{486}{\pi^2}L(\widetilde{F}_{108},2),
	\end{equation}
	\begin{equation}\label{iden5}
		\begin{split}
			&m_3\bigl(-216(18964+13149\sqrt[3]{3}+9117\sqrt[3]{9})\bigr)=\frac{729}{4\pi^2}L(F_{243},2),\\[4pt]
			&m_3\bigl(-108(37928\!-\!13149 \sqrt[3]{3} (1\!\pm\! i \sqrt{3})\!-\!9117\sqrt[3]{9} (1\!\mp\! i \sqrt{3}))\bigr)=\frac{5103}{4\pi^2}L(\widetilde{F}_{243},2),
		\end{split}
	\end{equation}
	\begin{equation}\label{iden6}
		m_3\bigl(6+3\sqrt[3]{2}-9\sqrt[3]{4}\pm 3i\sqrt{3}(\sqrt[3]{2}+3\sqrt[3]{4})\bigr)\!=\!\frac{81}{2\pi^2}L(G_{108},2),
	\end{equation}
	\begin{equation}\label{iden7}
		m_3\bigl(96-28\sqrt[3]{3}+36\sqrt[3]{9}\pm 4i\sqrt{3}(7\sqrt[3]{3}+9\sqrt[3]{9})\bigr)\!=\!\frac{243}{4\pi^2}L(G_{243},2),
	\end{equation}
	where $\alpha, \beta$ are the same as those appearing in \eqref{Samartconj2} and $F_N,\widetilde{F}_{N},G_N$ are normalized (i.e., with leading coefficient $1$) weight $2$ cusp forms of the form
	\[r\sum_{m,n\in\Z}\chi_{-3}(n)(lm+sn)q^{am^2+bmn+cn^2}\]
	in $\mathcal{S}_2(\Gamma_0(N))$. We make them clear in Table \ref{cuspforms} by listing the corresponding $r,l,s,a,b,c$.
\end{thm}
\begin{table}[ht]
	\centering
		\begin{tabular}{|ccccccc|ccccccc|}
			\hline
			Cusp forms  & $r$ & $l$  & $s$  & $a$ & $b$ & $c$ & Cusp forms  & $r$ & $l$  & $s$  & $a$ & $b$ & $c$ \\[0pt]
			\hline
			$F_{108}$ & $1/2$ & $0$ & $1$ & $27$ & $0$ & $1$ & $F_{243}$ & $1/4$ & $3$ & $2$ & $63$ & $3$ & $1$ \\[1pt]
			$\widetilde{F}_{108}$ & $1/8$ & $3$ & $4$ & $9$ & $6$ & $4$ & $\widetilde{F}_{243}$ & $1/28$ & $3$ & $14$ & $9$ & $3$ & $7$ \\[1pt]
			$F_{144}$ & $1/2$ & $0$ & $1$ & $12$ & $0$ & $1$ & $F_{441}$ & $1/4$ & $3$ & $2$ & $39$ & $3$ & $1$ \\[1pt]
			$\widetilde{F}_{144}$ & $1/2$ & $0$ & $1$ & $3$ & $0$ & $4$ & $G_{108}$ & $1/2$ & $1$ & $1$ & $4$ & $2$ & $1$ \\[1pt]
			$F_{225}$ & $1/4$ & $3$ & $2$ & $21$ & $3$ & $1$ & $G_{243}$ & $1/4$ & $1$ & $2$ & $7$ & $1$ & $1$ \\[1pt]
			\hline
		\end{tabular}
	\smallskip
	\caption{Cusp forms in Theorem \ref{mainresult}.}
	\label{cuspforms}
\end{table}
Since $6-6\sqrt[3]{2}+18\sqrt[3]{4}$ and $6+3\sqrt[3]{2}-9\sqrt[3]{4}\pm 3i\sqrt{3}(\sqrt[3]{2}+3\sqrt[3]{4})$ are the three roots of
	\[T^3-18T^2+756T-27000=0,\]
our result \eqref{iden6} can be seen as a supplement to Samart's identity \eqref{Samartresult}. Moreover, \eqref{iden4} and some linear combinations of modular forms yield:
\begin{cor}\label{Samartconj}
	The identities \eqref{Samartconj1} and \eqref{Samartconj2} are true.
\end{cor}

When $k=\sqrt[3]{729\pm405\sqrt{3}}$, guided by Beilinson's conjecture for curves over number fields,  we can also prove the following result similar to \eqref{mahlerasdet}.

\begin{thm}\label{detmahler}
	Consider $C_{\sqrt[3]{729\pm405\sqrt{3}}}$ as elliptic curves defined over $\Q(\sqrt{3})$ (see \eqref{EllipticcurveCk}). Then we have
	\begin{equation*}
		\begin{split}
			\left|\det{\begin{pmatrix}m_3(729+405\sqrt{3})& m_3(729-405\sqrt{3})\\
				m_3(729-405\sqrt{3})&4m_3(729+405\sqrt{3})\end{pmatrix}}\right| & =\frac{19683}{\pi^4}L(C_{\sqrt[3]{729\pm405\sqrt{3}}},2) \\
				& =\frac{243}{8}L''(C_{\sqrt[3]{729\pm405\sqrt{3}}},0).
		\end{split}
	\end{equation*}
\end{thm}

For $k\in \mathcal{K}_Q^\circ$, the zero locus of $Q_k(x,y)$ will intersect with $\mathbb{T}^2$. We cannot expect $m_3(k^3)$ to be related to  $L$-values of modular forms or elliptic curves in this case, since we cannot write $m_3(k^3)$ as a regulator integral over some \emph{closed} path and thus Beilinson's conjecture does not work. In \cite{Sam23}, Samart turned to the tempered polynomial family
\[\widetilde{Q}_k(x,y)=y^2+(x^2-kx)y+x.\]
Note that $m(Q_k)=m(\widetilde{Q}_k)$ because $(x^2y)^3Q_k(y/x^2,1/xy)=\widetilde{Q}_k(x^3,y^3)$. When $k\in \mathcal{K}_Q^\circ\cap\R=(-1,3)$, he proved that the zero locus of $\widetilde{Q}_k(x,y)$ intersects with $\mathbb{T}^2$ at
\[\left\{(e^{i\theta},\tilde{y}^\pm(e^{i\theta}))\mid\theta=0,\pm\cos^{-1}\left(\frac{k-1}{2}\right)\right\},\]
where $\tilde{y}^\pm(x)=-(x^2-kx)\Bigl(\frac{1}{2}\pm\sqrt{\frac{1}{4}-\frac{1}{x(x-k)^2}}\Bigr)$ with the square root chosen to be the principal branch. Clearly, one can factorize $\widetilde{Q}_k(x,y)$ as $(y-\tilde{y}^+(x))(y-\tilde{y}^-(x))$. Samart then introduced the modified Mahler measure
\begin{equation}\label{modifiedMM}
	\tilde{n}(k):=m(\widetilde{Q}_k)-\frac{3}{\pi}\int_{\cos^{-1}\left(\frac{k-1}{2}\right)}^{\pi}\log|\tilde{y}^+(e^{i\theta})|d\theta.
\end{equation}
This modification allowed him to interpret $\tilde{n}(k)$ as the regulator integral over a carefully chosen closed path on the Riemann surface associated to $\{(x,y)\in\C^2\mid\widetilde{Q}_k(x,y)=0\}$ for $k\in(-1,3)$. By using the modular unit parametrization for $\widetilde{Q}_2(x,y)=0$, he proved that
\[\tilde{n}(2)=-3L'(C_2,0).\]
Based on numerical evidences, he also made other conjectures for $\tilde{n}(k)$ with $k^3=1,2,\cdots,26$. Observe that when $k=\sqrt[3]{24}$, the elliptic curve
\[C_{\sqrt[3]{24}}:Y^2=X^3-15552X^2-8957952X-1289945088\]
has CM. The last result of this paper is the following identity that was conjectured by Samart in \cite[Table 2]{Sam23}.
\begin{thm}\label{modifiedMMasEllipticcurve}
	Let $\tilde{n}(k)$ be Samart's modified Mahler measure \eqref{modifiedMM}. Then we have
	\[\tilde{n}(\sqrt[3]{24})=-3L'(C_{\sqrt[3]{24}},0).\]
\end{thm}

This paper is organized as follows. In Section \ref{CPATA}, we briefly introduce the theory of complex multiplication. An algorithm is designed to search for all CM points in $\mathcal{F}'$ such that $h(\tau)\leqslant 3,h(\tau)h(3\tau)\leqslant 4$. In Section \ref{TRWCF}, we will prove Theorem \ref{mainresult} and Corollary \ref{Samartconj}. A transformation formula for general theta functions is used to verify that the modular forms in Table \ref{cuspforms} are indeed cusp forms for $\Gamma_0(N)$. In Section \ref{LTTLOIEC}, we construct the Beilinson regulator that relates $m_3(729\pm405\sqrt{3})$ to $L(C_{\sqrt[3]{729\pm405\sqrt{3}}},2)$. This can help us to prove Theorem \ref{detmahler}. Finally, according to Samart's hypergeometric formula for $\tilde{n}(k)$, we will prove Theorem \ref{modifiedMMasEllipticcurve}. 

\section{CM points and the algorithm}\label{CPATA}

Let $E$ be an elliptic curve over $\C$. Then $E=\C/\Z\oplus\Z\tau$ for some $\tau\in\mathcal{H}$ that is unique up to an action by $\SL(2,\Z)$. Recall that the $j$-invariant of $E$ is defined by
\[j(\tau)=\frac{1}{q}+744+196884q+21493760q^2+\cdots,\]
where $q=e^{2\pi i\tau}$. It is well known that $E$ has CM if and only if $\tau$ is a CM point. Furthermore, $j(\tau)$ are algebraic integers of degree $h(\tau)$ if $\tau$ takes CM points, these algebraic integers are called \emph{singular moduli}. When $\tau=[a,b,c]$ is a CM point, we call the singular modulus $j(\tau)$ is of \emph{discriminant} $b^2-4ac$. This is a convenience borrowed from \cite{BHK20}. For every negative integer $D$ with $D\equiv 0$ or $1\mod 4$, there are exactly $h(D)$ different singular moduli of discriminant $D$ which form a full Galois orbit over $\Q$. Let $\mathcal{F}=\{\tau\in\mathcal{H}\mid-1/2\leqslant \re(\tau)\leqslant 1/2,|\tau|\geqslant 1\}$ be the fundamental domain of $\SL(2,\Z)$. We can find $h(D)$ CM points $\tau_1,\cdots,\tau_{h(D)}\in\mathcal{F}$ with discriminant $D$ that are in different $\SL(2,\Z)$-orbits. Then $j(\tau_1), \cdots, j(\tau_{h(D)})$ are all the different singular moduli of discriminant $D$. Moreover, 
\[\prod_{i=1}^{h(D)}\bigl(X-j(\tau_i)\bigr)\]
is the monic polynomial in $\Z[X]$ that makes $j(\tau_1), \cdots, j(\tau_{h(D)})$ algebraic integers (see, for instance, \cite{Cox,Sil94}).

For a general congruence subgroup $\Gamma_0(N),N\geqslant 1$, its modular functions also take algebraic values at CM points under suitable conditions. Let $\{\gamma_1,\cdots,\gamma_r\}$ be a set of right coset representatives of $\Gamma_0(N)$ in $\SL(2,\Z)$. It is known that
\[r=\bigl[\SL(2,\Z):\Gamma_0(N)\bigr]=N\prod_{\substack{p\mid N\\ p\ \text{prime}}}\Bigl(1+\frac{1}{p}\Bigr),\]
and there exists a polynomial $\Phi_N(X,Y)\in\Z[X,Y]$ (the so-called \emph{modular equation} for $\Gamma_0(N)$) such that
\begin{equation}\label{modularequation}
	\Phi_N(X,j(\tau))=\prod_{i=1}^{r}\bigl(X-j(N\gamma_i\tau)\bigr).
\end{equation}
\begin{pro}[{\cite[Proposition 12.7]{Cox}}]\label{Coxtheorem}
	Let $f(\tau)$ be a modular function for $\Gamma_0(N)$ whose $q$-expansion has rational coefficients. Then:
	\begin{enumerate}
		\item $f(\tau)\in\Q(j(\tau),j(N\tau))$.
		\item Assume in addition that $f(\tau)$ is holomorphic on $\mathcal{H}$, and let $\tau_0\in\mathcal{H}$. If
		\[\frac{\partial\Phi_N}{\partial X}\bigl(j(N\tau_0),j(\tau_0)\bigr)\neq 0,\]
		then $f(\tau_0)\in \Q(j(\tau_0),j(N\tau_0))$.
	\end{enumerate}
\end{pro}

In the case when $N=3$, we have $r=3\cdot(1+\frac{1}{3})=4$. Let $S=\begin{pmatrix}0 & -1\\ 1 & 0\end{pmatrix},T=\begin{pmatrix}1 & 1\\ 0 & 1\end{pmatrix}$ be the generators of $\SL(2,\Z)$, then
\begin{equation}\label{RCRforGamma03}
	\gamma_1=I_2=\begin{pmatrix}1 & 0\\ 0 &1\end{pmatrix},\quad \gamma_2=S,\quad \gamma_3=ST,\quad \gamma_4=ST^{-1}
\end{equation}
form a set of right coset representatives of $\Gamma_0(3)$ in $\SL(2,\Z)$. These matrices can also transform the fundamental domain $\mathcal{F}$ to cover $\mathcal{F}'$:
\begin{equation}\label{coverofFprime}
	\mathcal{F}'=\bigcup_{i=1}^4\gamma_i\mathcal{F}.
\end{equation}
See Figure \ref{Fprime_figure} for the picture.

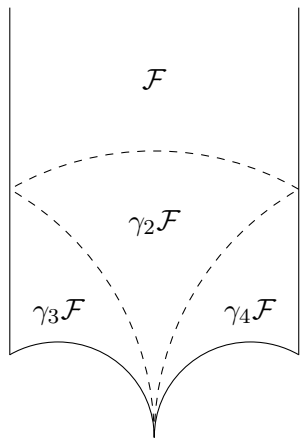
\begin{figure}[h]
	\centering
	\begin{tikzpicture}[scale=1.9]
		\draw (-1,0.575) -- (-1,3);
		\draw (1,0.575) -- (1,3);
		\draw[dashed] (1,1.732) arc(60:120:2);
		\node at (90:2.5) {{\small$\mathcal{F}$}};
		\node at (90:1.5) {{\small$\gamma_2\mathcal{F}$}};
		\node at (127:1.1) {{\small$\gamma_3\mathcal{F}$}};
		\node at (53:1.1) {{\small$\gamma_4\mathcal{F}$}};
		\draw[dashed] (0,0) arc(0:60:2);
		\draw[dashed] (0,0) arc(180:120:2);
		\draw (0,0) arc(0:120:0.6666);
		\draw (0,0) arc(180:60:0.6666);
	\end{tikzpicture}
	\caption{$\mathcal{F}'$ and its covering \eqref{coverofFprime}.}
	\label{Fprime_figure}
\end{figure}

Since $\eta(\tau)=e^\frac{2\pi i\tau}{24}\prod\limits_{n=1}^{\infty}(1-q^n)$, the modular function $t(\tau)=27+\bigl(\eta(\tau)/\eta(3\tau)\bigr)^{12}$ for $\Gamma_0(3)$ in Theorem \ref{Villegasthm} has $q$-expansion
\[t(\tau)=\frac{1}{q}+15+54q-76q^2-243q^3+1188q^4+\cdots,\]
the coefficients are rational integers. It is also easily seen that $t(\tau)$ is holomorphic on $\mathcal{H}$ since $\eta(\tau)$ has no zeros on $\mathcal{H}$. Thus, we can apply Proposition \ref{Coxtheorem} to $t(\tau)$ and prove the following result that is similar to \cite[Theorem 2.2]{TGW22}.
\begin{thm}\label{boundoft}
	Let $\tau_0\in\mathcal{H}$ be a CM point. If $j(3\tau_0)\neq j(3\gamma_i\tau_0)$ for $i=2,3,4$, where $\gamma_i$ are as in \eqref{RCRforGamma03}, then $t(\tau_0)$ is an algebraic number with degree no more than $h(\tau_0)h(3\tau_0)$.
\end{thm}

The above theorem implies that, in order to obtain some interesting CM points that keep the degrees of $t(\tau)$ not too high, we can search for CM points in $\mathcal{F}'$ with class numbers relatively small. In this work, we focus
our attention on CM points with class numbers $\leqslant 3$ and will use them to search for $t(\tau)$ that have degrees $\leqslant 4$ as algebraic numbers. To achieve this, our first task is to determine discriminants with some small class numbers. It is well known that for each positive integer $n$, only finitely many negative discriminants $D\equiv 0$ or $1\mod 4$ have $h(D)=n$. We list all negative discriminants that have class numbers $1,2$ and $3$ as follows.
\begin{equation*}
	\begin{split}
		h(D)=1\!\iff\! D = & -3,-4,-7,-8,-11,-12,-16,-19,-27,-28,-43,-67,-163;\\
		h(D)=2\!\iff\! D = & -15,-20,-24,-32,-35,-36,-40,-48,-51,-52,-60,-64,-72,\\
		& -75,-88,-91,-99,-100,-112,-115,-123,-147,-148,-187,-232,\\
		&-235,-267,-403,-427;\\
		h(D)=3\!\iff\! D = & -23,-31,-44,-59,-76,-83,-92,-107,-108,-124,-139,-172,\\
		&-211,-243,-268,-283,-307,-331,-379,-499,-547,-643,-652,\\
		&-883,-907.
	\end{split}
\end{equation*}
Just like what Huber, Schultz and Ye did in \cite[Algorithm 4.1]{HSY18}, we can apply the following algorithm to determine all CM points $\tau\in \mathcal{F}'$ such that $h(\tau)\leqslant 3, h(\tau)h(3\tau)\leqslant 4$.
\begin{algo}\label{Searchcmpoints} For each discriminant $D$ listed above, perform the following operations.
	\begin{enumerate}
		\item Determine all CM points $\tau=[a,b,c]$ in $\mathcal{F}$ with discriminant $D$  by solving the system
		\[\begin{cases}
			a,b,c\in\Z,a>0,\\
			\gcd(a,b,c)=1,\\
			b^2-4ac=D, \\
			|b|\leqslant a\leqslant c.
		\end{cases}\]
		\item Use the right coset representatives \eqref{RCRforGamma03} to translate these points obtained in (1) to get a full list of $CM$ points in $\mathcal{F}'$ with discriminant $D$.
		\item For each point $\tau$ obtained in (2), verify whether $h(\tau)h(3\tau)\leqslant 4$. If it passes this test, then output $\tau$.
	\end{enumerate}
\end{algo}

We implement Algorithm \ref{Searchcmpoints} in a \textsf{Mathematica} notebook. Thanks to the built-in function \texttt{DedekindEta[]}, $t(\tau)$ can be calculated to any given precision. This can help us to filter out the points that make $\sqrt[3]{t(\tau)}\in \C-\mathcal{K}^\circ_Q$ from the outputs of Algorithm \ref{Searchcmpoints}. We list these lucky ones and the corresponding numerical approximations (accurate to five decimal places) of $t(\tau)$ in Table \ref{CMpointsandttau}. According to Theorem \ref{boundoft}, all values in Table \ref{CMpointsandttau} are in fact approximations of algebraic numbers with degrees $\leqslant 4$.

\begin{table}[ht]
	\centering
	{\tiny
		\begin{tabular}{|cc|cc|cc|}
			\hline
			$\tau$  & $t(\tau)$ & $\tau$  & $t(\tau)$  & $\tau$ & $t(\tau)$ \\[0pt]
			\hline
			$\mathit{[1,-1,1]}$ & $\mathit{-216.00000}$ & $[2,0,1]$ & $100.64395$ & $[9,2,1]$ & $25.99494\!+\!1.63455i$ \\[1pt]
			$[1,-1,2]$ & $-4056.94536$ & $[3,2,1]$ & $4.00000\!+\!14.14213i$ & $[9,-2,1]$ & $25.99494\!-\!1.63455i$ \\[1pt]
			$[2,1,1]$ & $15.01873\!+\!62.96451i$ & $[3,-2,1]$ & $4.00000\!-\!14.14213i$ & $[3,2,3]$ & $-171.99494\!+\!323.63455i$ \\[1pt]
			$[2,-1,1]$ & $15.01873\!-\!62.96451i$ & $\mathit{[1,0,3]}$ & $\mathit{53267.29623}$ & $[3,-2,3]$ & $-171.99494\!-\!323.63455i$ \\[1pt]
			$[1,-1,3]$ & $-33491.14467$ & $\mathit{[4,2,1]}$ & $\mathit{15.35188\!+\!11.56864i}$ & $[9,0,1]$ & $28.39230$ \\[1pt]
			$[3,1,1]$ & $32.00000\!+\!26.53299i$ & $\mathit{[4,-2,1]}$ & $\mathit{15.35188\!-\!11.56864i}$ & $\mathit{[12,0,1]}$ & $\mathit{27.51942}$ \\[1pt]
			$[3,-1,1]$ & $32.00000\!-\!26.53299i$ & $\mathit{[3,0,1]}$ & $\mathit{54.00000}$ & $\mathit{[3,0,4]}$ & $\mathit{1430.48057}$ \\[1pt]
			$[1,-1,5]$ & $-885464.77774$ & $[1,0,4]$ & $286766.31332$ & $[15,0,1]$ & $27.21952$ \\[1pt]
			$[5,1,1]$ & $30.24531\!+\!7.39216i$ & $[4,0,1]$ & $40.31662$ & $[3,0,5]$ & $3347.78047$ \\[1pt]
			$[5,-1,1]$ & $30.24531\!-\!7.39216i$ & $[5,2,1]$ & $20.68502\!+\!7.84745i$ & $[18,0,1]$ & $27.10102$ \\[1pt]
			$\mathit{[1,-1,7]}$ & $\mathit{-12288728.98398}$ & $[5,-2,1]$ & $20.68502\!-\!7.84745i$ & $[9,0,2]$ & $36.89897$ \\[1pt]
			$\mathit{[7,1,1]}$ & $\mathit{28.49199\!+\!2.87797i}$ & $[1,0,7]$ & $16580645.98812$ & $[3,-3,2]$ & $-43.68691$ \\[1pt]
			$\mathit{[7,-1,1]}$ & $\mathit{28.49199\!-\!2.87797i}$ & $[7,0,1]$ & $29.99744$ & $[3,-3,5]$ & $-1754.59081$ \\[1pt]
			$[1,-1,11]$ & $-884736728.99977$ & $[8,2,1]$ & $25.50721\!+\!2.35941i$ & $\mathit{[3,-3,7]}$ & $\mathit{-8666.91614}$ \\[1pt]
			$[11,1,1]$ & $27.37486\!+\!0.66585i$ & $[8,-2,1]$ & $25.50721\!-\!2.35941i$ & $[3,-3,11]$ & $-110618.99341$ \\[1pt]
			$[11,-1,1]$ & $27.37486\!-\!0.66585i$ & $[9,1,1]$ & $27.72446\!+\!1.31876i$ & $\mathit{[3,-3,13]}$ & $\mathit{-326618.99776}$ \\[1pt]
			$[1,-1,17]$ & $-147197952728.99999$ & $[9,-1,1]$ & $27.72446\!-\!1.31876i$ & $[3,-3,23]$ & $-27000026.99997$ \\[1pt]
			$[17,1,1]$ & $27.06887\!+\!0.11984i$ & $[3,1,3]$ & $260.27553\!+\!424.63897i$ & $[9,-9,5]$ & $-18.97825$ \\[1pt]
			$[17,-1,1]$ & $27.06887\!-\!0.11984i$ & $[3,-1,3]$ & $260.27553\!-\!424.63897i$ & $\mathit{[27,0,1]}$ & $\mathit{27.01369}$ \\[1pt]
			$[1,-1,41]$ & $-262537412640768728.99999$ & $[6,2,1]$ & $23.30495\!+\!5.19349i$ & $\mathit{[9,6,4]}$ & $\mathit{-4.50684\!+\!31.29187i}$ \\[1pt]
			$[41,1,1]$ & $27.00056\!+\!0.00098i$ & $[6,-2,1]$ & $23.30495\!-\!5.19349i$ & $\mathit{[9,-6,4]}$ & $\mathit{-4.50684\!-\!31.29187i}$ \\[1pt]
			$[41,-1,1]$ & $27.00056\!-\!0.00098i$ & $[3,2,2]$ & $-39.30495\!+\!93.19349i$ & $\mathit{[9,3,7]}$ & $\mathit{130.50002\!+\!199.64657i}$ \\[1pt]
			$[1,0,1]$ & $550.59223$ & $[3,-2,2]$ & $-39.30495\!-\!93.19349i$ & $\mathit{[9,-3,7]}$ & $\mathit{130.50002\!-\!199.64657i}$ \\[1pt]
			$[2,2,1]$ & $-10.59223$ & $[6,0,1]$ & $31.63246$ &  & \\[1pt]
			$[1,0,2]$ & $7243.35604$ & $[3,0,2]$ & $184.36753$ &  & \\[1pt]
			\hline
		\end{tabular}
	}
	\smallskip
	\caption{CM points and the numerical approximations of $t(\tau)$.}
	\label{CMpointsandttau}
\end{table}

In general, each CM point $\tau_0$ in Table \ref{CMpointsandttau} will produce an identity of the form
\[m_3(t(\tau_0))=\frac{c_{\tau_0}}{\pi^2}L(f_{\tau_0},2),\]
with $f_{\tau_0}$ a normalized cusp form and $c_{\tau_0}$ a real quadratic number such that $c_{\tau_0}^2\in\Q$. To limit the length of this paper, we only deal with those points that make $c_{\tau_0}\in\Q$, and they are italicized in Table \ref{CMpointsandttau}. Note that $[1,-1,1]=\frac{1+i\sqrt{3}}{2},[3,0,1]=\frac{i\sqrt{3}}{3}$ and $[27,0,1]=\frac{i\sqrt{3}}{9}$ correspond to the already proven results \cite{RV98, Rog11, Sam14} for $m_3(-216),m_3(54)$ and $m_3\bigl(6-6\sqrt[3]{2}+18\sqrt[3]{4}\bigr)$, respectively.

\section{The relations with cusp forms}\label{TRWCF}

In order to recognize the lattice sums appearing in \eqref{Villegasformula} as $L$-values of modular forms, we need some facts about theta functions associated to lattices. Recall that an \emph{even} lattice $L$ of rank $n$ is the submodule $\Z^n$ of $\R^n$ equipped with a nondegenerate quadratic form
\[\mathcal{Q}(X)=\frac{1}{2}X^tAX,\quad X\in\Z^n,\]
where $A$ is an \emph{even} matrix of rank $n$, that is, $A$ is an $n\times n$ symmetric matrix with integer entries and even integer diagonals. Obviously, this definition makes $\mathcal{Q}(X)\in\Z$ for every $X\in\Z^n$. The \emph{level} of $L$ is defined to be the smallest positive integer $N$ such that $NA^{-1}$ is an even matrix. We also define $L^*=\{Y\in\R^n\mid X^{t}AY\in\Z,\forall X\in\Z^n\}$ to be the \emph{dual lattice} of $L$.

\begin{pro}[{\cite[Corollary 14.3.16]{CS17}}]\label{thetafunction}
	Let $L$ be a positive definite even lattice of \emph{even} rank $n$, level $N$, and quadratic form $\mathcal{Q}(X)=\frac{1}{2}X^tAX$. Assume that $A^{-1}=(b_{i,j})_{1\leqslant i,j\leqslant n}$ and let $P(x_1,\cdots,x_n)$ be a homogeneous polynomial of degree $(k-n)/2$ such that $\Delta_\mathcal{Q}(P)=0$, where
	\[\Delta_\mathcal{Q}=\sum_{1\leqslant i,j\leqslant n}b_{i,j}\frac{\partial^2}{\partial x_i\partial x_j}.\]
	Then for all $Y\in L^*$, the \emph{theta function}
	\[\Theta(P,L,Y;\tau):=\sum_{X\in\Z^n}P(X+Y)q^{\mathcal{Q}(X+Y)}\]
	is in $\mathcal{M}_{k/2}(\Gamma(N))$. In addition, if $k>n$, then $\Theta$ is also a cusp form.
\end{pro}

It is clear from the definition that $\Theta(P,L,Y;\tau)$ only depends on the class of $Y$ in $L^*/\Z^n$. Let $k$ be an integer. Recall that the weight $k$ slash operator $\big{|}_k$ is given by
\[(f\big{|}_k\gamma)(\tau):=(c\tau+d)^{-k}f(\gamma\tau),\quad\forall\gamma=\begin{pmatrix}a & b\\ c & d\end{pmatrix}\in\SL(2,\Z)\text{ and }f:\mathcal{H}\to\C.\]
For elements in $\Gamma_0(N)$, we also have the following transformation
formula.

\begin{pro}[{\cite[Chapter IX, \S 4, Theorem 5]{Sch74}}]\label{transformationformula}
	Let the assumptions of Proposition \ref{thetafunction} hold. For $\gamma=\begin{pmatrix}a & b\\ c & d\end{pmatrix}\in \Gamma_0(N)$, we have
	\[\bigl(\Theta(P,L,Y;\cdot)\big{|}_{k/2}\gamma\bigr)(\tau)=v(d)e^{2\pi iab\mathcal{Q}(Y)}\Theta(P,L,aY;\tau),\]
	where $v(d)=\left(\frac{(-1)^{n/2}\det{A}}{d}\right)$ if $d>0$, and $v(d)=(-1)^{n/2}\left(\frac{(-1)^{n/2}\det{A}}{-d}\right)$ if $d<0$.
\end{pro}
With these tools in hand, we can now prove the identities in Theorem \ref{mainresult}. We first deal with the cases when the class numbers are equal to $1$.

\subsection{Proofs of the cases when \texorpdfstring{$h(\tau)=1$}{h(tau)=1}}

\begin{proof}[Proof of \eqref{iden4}]
	Since $17766+14094\sqrt[3]{2}+11178\sqrt[3]{4}=53267.29623\cdots$, by examining Table \ref{CMpointsandttau}, we observe that $\tau_0=[1,0,3]=i\sqrt{3}$ seems to be the candidate such that
	\begin{equation}\label{tauforiden4}
		t(\tau_0)=17766+14094\sqrt[3]{2}+11178\sqrt[3]{4}.
	\end{equation}
	This can be rigorously proved by using the identity \cite[\S 2]{Sam14}
	\begin{equation}\label{Samartformula}
		j(\tau)=\frac{t(\tau)(t(\tau)+216)^3}{(t(\tau)-27)^3}
	\end{equation}
	together with the fact that $j(i\sqrt{3})$ is a rational integer because $h(i\sqrt{3})=h(-12)=1$. A numerical calculation shows that $j(i\sqrt{3})=54000$, and thus \eqref{tauforiden4} is confirmed. It turns out that $\sqrt[3]{t(i\sqrt{3})}\notin\mathcal{K}_Q^\circ$ since $\mathcal{K}_Q\cap\R=[-1,3]$. Taking $\tau=\tau_0$ in Theorem \ref{Villegasthm} then yields
	\begin{equation*}
		\begin{split}
			m_3\bigl(17766+14094\sqrt[3]{2}+11178\sqrt[3]{4}\bigr) & =\frac{243}{4\pi^2}\underset{m,n\in\Z}{{\sum}'}\frac{\chi_{-3}(n)n}{(27m^2+n^2)^2}\\
			& = \frac{243}{2\pi^2}L(F_{108},2),
		\end{split}
	\end{equation*}
	where
	\begin{equation*}
		\begin{split}
			F_{108}(\tau) & =\frac{1}{2}\underset{m,n\in\Z}{{\sum}}\chi_{-3}(n)nq^{27m^2+n^2}\\
			& = \frac{1}{2}\left(\underset{m,n\in\Z}{{\sum}}(3n+1)q^{27m^2+(3n+1)^2}-\underset{m,n\in\Z}{{\sum}}(3n+2)q^{27m^2+(3n+2)^2}\right)\\
			& = \frac{3}{2}\left(\underset{m,n\in\Z}{{\sum}}(n+\tfrac{1}{3})q^{27m^3+9(n+\frac{1}{3})^2}-\underset{m,n\in\Z}{{\sum}}(n+\tfrac{2}{3})q^{27m^3+9(n+\frac{2}{3})^2}\right).
		\end{split}
	\end{equation*}
	Let $L$ be the rank $2$ lattice with quadratic form $\mathcal{Q}(x_1,x_2)=27x_1^2+9x_2^2$. The Gram matrix is $A=\begin{pmatrix}54 & 0\\ 0 & 18\end{pmatrix}$ and thus the level $N=108$. We have
	\[F_{108}(\tau)=\frac{3}{2}\bigl(\Theta(P,L,Y_1;\tau)-\Theta(P,L,Y_2;\tau)\bigr),\]
	where $P(x_1,x_2)=x_2$ and $Y_1=(0,1/3)^t, Y_2=(0,2/3)^t\in L^*$. Proposition \ref{thetafunction} immediately implies that $F_{108}(\tau)\in\mathcal{S}_{2}(\Gamma(108))$. In fact, we can prove that $F_{108}(\tau)\in\mathcal{S}_{2}(\Gamma_0(108))$ by verifying
	\begin{equation}\label{verifymodular}
		(F_{108}\big{|}_2\gamma)(\tau)=F_{108}(\tau),\quad \forall\gamma\in\Gamma_0(108).
	\end{equation}
	Take $\gamma=\begin{pmatrix}a & b\\ c & d\end{pmatrix}\in \Gamma_0(108)$. We have $(a,d)\equiv(\pm1,\pm1)\ \ (\mathrm{mod}\ 6)$ because $ad\equiv 1\ \ (\mathrm{mod}\ 108)$. Also note that $\det{A}=972$ and
	\[\left(\frac{-972}{d}\right)=\begin{cases}
		1, & \text{if }d\equiv 1\ \ (\mathrm{mod}\  6),\\
		-1, & \text{if }d\equiv -1\ \ (\mathrm{mod}\  6),\\
		0, & \text{otherwise}.
	\end{cases}\]
	If $(a,d)\equiv(1,1)\ \ (\mathrm{mod}\ 6)$ and $d>0$, then, according to Proposition \ref{transformationformula}, we have
	\begin{equation*}
		\begin{split}
			(F_{108}\big{|}_2\gamma)(\tau) & = \frac{3}{2}\bigl(\bigl(\Theta(P,L,Y_1;\cdot)\big{|}_2\gamma\bigr)(\tau)-\bigl(\Theta(P,L,Y_2;\cdot)\big{|}_2\gamma\bigr)(\tau)\bigr) \\[1pt]
			& = \frac{3}{2}\left(\frac{-972}{d}\right)\bigl(e^{2\pi i ab}\Theta(P,L,aY_1;\tau)-e^{8\pi i ab}\Theta(P,L,aY_2;\tau)\bigr)\\[1pt]
			& = \frac{3}{2}\bigl(\Theta(P,L,aY_1;\tau)-\Theta(P,L,aY_2;\tau)\bigr)\\[1pt]
			& = \frac{3}{2}\bigl(\Theta(P,L,Y_1;\tau)-\Theta(P,L,Y_2;\tau)\bigr)\\[1pt]
			& = F_{108}(\tau).
		\end{split}
	\end{equation*}
	The fourth equality holds because $Y_1-aY_1, Y_2-aY_2 \in \Z^2$. If $(a,d)\equiv(-1,-1)\ \ (\mathrm{mod}\ 6)$ and $d>0$, we have $Y_2-aY_1,Y_1-aY_2\in\Z^2$, thus
	\begin{equation*}
		\begin{split}
			(F_{108}\big{|}_2\gamma)(\tau) & = -\frac{3}{2}\bigl(\Theta(P,L,aY_1;\tau)-\Theta(P,L,aY_2;\tau)\bigr)\\[1pt]
			& = -\frac{3}{2}\bigl(\Theta(P,L,Y_2;\tau)-\Theta(P,L,Y_1;\tau)\bigr)\\[1pt]
			& = F_{108}(\tau).
		\end{split}
	\end{equation*}
	In the cases when $(a,d)\equiv(\pm1,\pm1)\ \ (\mathrm{mod}\ 6)$ and $d<0$, \eqref{verifymodular} also holds because
	\[-\left(\frac{-972}{-d}\right)=\begin{cases}
		1, & \text{if }d\equiv 1\ \ (\mathrm{mod}\ 6),\\
		-1, & \text{if }d\equiv -1\ \ (\mathrm{mod}\ 6).
	\end{cases}\]
	This completes our verification of \eqref{verifymodular}.

	To prove the identities $m_3(\alpha\pm i\beta)\!=\!\frac{486}{\pi^2}L(\widetilde{F}_{108},2)$, we calculate that
	\[\alpha=15.35188\cdots,\quad\beta=11.56864\cdots.\]
	This time, according to Table \ref{CMpointsandttau}, we need to prove that
	\[t([4,\pm2,1])=t\biggl(\frac{\mp1+i\sqrt{3}}{4}\biggr)=\alpha\pm i\beta.\]
	Indeed, this can be confirmed by the facts that $h([4,\pm2,1])=h(-12)=1$ and $j([4,\pm2,1])=54000$. It follows from Theorem \ref{Villegasthm} that
	\begin{equation*}
		\begin{split}
			m_3(\alpha\pm i\beta) & =\frac{243}{4\pi^2}\underset{m,n\in\Z}{{\sum}'}\frac{\chi_{-3}(n)(3m+4n)}{(9m^2+6mn+4n^2)^2}\\
			& = \frac{486}{\pi^2}L(\widetilde{F}_{108},2),
		\end{split}
	\end{equation*}
	where
	\begin{equation*}
		\begin{split}
			\widetilde{F}_{108}(\tau) & =\frac{1}{8}\sum_{m,n\in\Z}\chi_{-3}(n)(3m+4n)q^{9m^2+6mn+4n^2}\\
			& = \frac{3}{8}\left(\underset{m,n\in\Z}{{\sum}}(m+4(n+\tfrac{1}{3}))q^{9m^3+18m(n+\frac{1}{3})+36(n+\frac{1}{3})^2}\right.\\
			&\quad\quad\quad\quad\quad\quad\quad\quad\quad\quad\quad\left.-\underset{m,n\in\Z}{{\sum}}(m+4(n+\tfrac{2}{3}))q^{9m^3+18m(n+\frac{2}{3})+36(n+\frac{2}{3})^2}\right).
		\end{split}
	\end{equation*}
	Similarly, we can write $\widetilde{F}_{108}(\tau)$ as $\frac{3}{8}\bigl(\Theta(P_1,L_1,Y_1;\tau)-\Theta(P_1,L_1,Y_2;\tau)\bigr)$ with $L_1$ the level $108$ lattice associated to $\mathcal{Q}_1(x_1,x_2)=9x_1^2+18x_1x_2+36x_2^2$ and $P_1(x_1,x_2)=x_1+4x_2$. A completely parallel argument can be made to show that $\widetilde{F}_{108}(\tau)\in\mathcal{S}_{2}(\Gamma_0(108)).$
\end{proof}

Now, we can provide a proof for Samart's conjectural identities \eqref{Samartconj1} and \eqref{Samartconj2}.

\begin{proof}[Proof of Corollary \ref{Samartconj}]
	According to the functional equation
	\[\biggl(\frac{\sqrt{N}}{2\pi}\biggr)^s\Gamma(s)L(f,s)=\pm\biggl(\frac{\sqrt{N}}{2\pi}\biggr)^{2-s}\Gamma(2-s)L(f,2-s)\]
	for $L$-functions of newforms of weight $2$ and level $N$, we have
	\[L'(f_{108},0)=\frac{27}{\pi^2}L(f_{108},2),\quad L'(f_{36},0)=\frac{9}{\pi^2}L(f_{36},2),\quad L'(f_{27},0)=\frac{27}{4\pi^2}L(f_{27},2).\]
	Thus, one can rewrite the right-hand sides of \eqref{Samartconj1} and \eqref{Samartconj2} as
 	\begin{equation*}
		\begin{split}
			\frac{3}{2}(L'(f_{108},0)+3L'(f_{36},0)+3L'(f_{27},0)) & =\frac{81}{2\pi^2}L(f_{108},2)+\frac{81}{2\pi^2}L(f_{36},2)+\frac{243}{8\pi^2}L(f_{27},2)\\
			& = \frac{243}{2\pi^2}\biggl(\frac{1}{3}L(f_{108},2)+\frac{1}{3}L(f_{36},2)+\frac{1}{4}L(f_{27},2)\biggr),\\
			\frac{3}{2}(L'(f_{108},0)+3L'(f_{36},0)-6L'(f_{27},0)) & =\frac{81}{2\pi^2}L(f_{108},2)+\frac{81}{2\pi^2}L(f_{36},2)-\frac{243}{4\pi^2}L(f_{27},2)\\
			& = \frac{486}{\pi^2}\biggl(\frac{1}{12}L(f_{108},2)+\frac{1}{12}L(f_{36},2)-\frac{1}{8}L(f_{27},2)\biggr).
		\end{split}
	\end{equation*}
	By \eqref{iden4}, it is hence enough to prove that
	\begin{equation*}
		\begin{split}
			\frac{1}{3}L(f_{108},2)+\frac{1}{3}L(f_{36},2)+\frac{1}{4}L(f_{27},2) & = L(F_{108},2), \\
			\frac{1}{12}L(f_{108},2)+\frac{1}{12}L(f_{36},2)-\frac{1}{8}L(f_{27},2) & = L(\widetilde{F}_{108},2).
		\end{split}
	\end{equation*}
	For this, we can first write $F_{108}(\tau)$ and $\widetilde{F}_{108}(\tau)$ as
	\begin{equation*}
		\begin{split}
			F_{108}(\tau) & = \frac{1}{3}f_{108}(\tau)+\frac{1}{3}f_{36}(\tau)+\frac{1}{3}f_{27}(\tau)-\frac{4}{3}f_{27}(4\tau),\\
			\widetilde{F}_{108}(\tau) & = \frac{1}{12}f_{108}(\tau)+\frac{1}{12}f_{36}(\tau)-\frac{1}{6}f_{27}(\tau)+\frac{2}{3}f_{27}(4\tau) 
		\end{split}
	\end{equation*}
	by using the fact that the Sturm bound for $\mathcal{M}_2(\Gamma_0(108))$ is $36$. Then
	\begin{equation*}
		\begin{split}
			L(F_{108},2) & = \frac{1}{3}L(f_{108},2)+\frac{1}{3}L(f_{36},2)+\frac{1}{3}L(f_{27},2)-\frac{4}{3}\cdot\frac{1}{4^2}L(f_{27},2) \\
			& = \frac{1}{3}L(f_{108},2)+\frac{1}{3}L(f_{36},2)+\frac{1}{4}L(f_{27},2),\\
			L(\widetilde{F}_{108},2) & = \frac{1}{12}L(f_{108},2)+\frac{1}{12}L(f_{36},2)-\frac{1}{6}L(f_{27},2)+\frac{2}{3}\cdot\frac{1}{4^2}L(f_{27},2) \\
			& = \frac{1}{12}L(f_{108},2)+\frac{1}{12}L(f_{36},2)-\frac{1}{8}L(f_{27},2),
		\end{split}
	\end{equation*}
	as desired.
\end{proof}

In the proofs of the remaining identities of Theorem \ref{mainresult}, we will give the CM points directly and omit the constructions of theta functions. They are all similar and not hard to do.

\begin{proof}[Proof of \eqref{iden5}]
	We just need to show that
	\begin{equation*}
		t([1,-1,7]) = t\biggl(\frac{1+3i\sqrt{3}}{2}\biggr)=-216(18964+13149\sqrt[3]{3}+9117\sqrt[3]{9}),
	\end{equation*}
	\begin{equation*}
		t([7,\pm1,1])  = t\biggl(\frac{\mp1+3i\sqrt{3}}{14}\biggr)=-108\bigl(37928\!-\!13149 \sqrt[3]{3} (1\!\pm\! i \sqrt{3})\!-\!9117\sqrt[3]{9} (1\!\mp\! i \sqrt{3})\bigr).
	\end{equation*}
	These equalities are all the consequences of $h([1,-1,7])=h([7,\pm1,1])=h(-27)=1$,
	\[j([1,-1,7])=j([7,\pm1,1])=-12288000,\]
	and \eqref{Samartformula}.
\end{proof}

For CM points with class numbers greater than $1$, the values of $j(\tau)$ are no longer rational integers. However, according to the theory of complex multiplication mentioned in Section \ref{CPATA}, these values are always determinable.

\subsection{Proofs of the cases when \texorpdfstring{$h(\tau)=2,3$}{h(tau)=2,3}}

\begin{proof}[Proofs of \eqref{iden1}, \eqref{iden2} and \eqref{iden3}]
	We can prove these identities by establishing the following equalities:
	\begin{equation}\label{t3-37}
		t([3,-3,7]) = t\biggl(\frac{3+5i\sqrt{3}}{6}\biggr)=-4320-1944\sqrt{5},
	\end{equation}
	\begin{equation*}
		t([3,-3,13]) = t\biggl(\frac{3+7i\sqrt{3}}{6}\biggr)=-163296-35640\sqrt{21},
	\end{equation*}
	\begin{equation*}
		t([3,0,4]) = t\biggl(\frac{2i\sqrt{3}}{3}\biggr)=729+405\sqrt{3},\quad t([12,0,1]) = t\biggl(\frac{i\sqrt{3}}{6}\biggr)=729-405\sqrt{3}.
	\end{equation*}
	Since $h([3,-3,7])=h([3,-3,13])=h([3,0,4])=h([12,0,1])=2$, the values of $j(\tau)$ at these points are all algebraic integers of degree $2$. In fact, we have
	\begin{equation}\label{j3-37}
		j([3,-3,7])=-327201914880+146329141248\sqrt{5},
	\end{equation}
	\begin{equation*}
		j([3,-3,13])=-17424252776448000+3802283679744000\sqrt{21},
	\end{equation*}
	\begin{equation*}
		j([3,0,4])=1417905000-818626500\sqrt{3},\quad j([12,0,1])=1417905000+818626500\sqrt{3}.
	\end{equation*}
	In the following, we will derive \eqref{t3-37} by proving \eqref{j3-37} and then using \eqref{Samartformula}. The others can be done in the same manner.
	
	Note that the discriminant of $\tau_1=[3,-3,7]$ is $-75$, one can choose $\tau_2$ to be the CM point with discriminant $-75$ that is not in the same $\SL(2,\Z)$-orbit with $\tau_1$, for instance, we choose $\tau_2=[1,-1,19]=\frac{1+5i\sqrt{3}}{2}$. Then $j(\tau_1)$ and $j(\tau_2)$ are all the $2$ different singular moduli of discriminant $-75$, so the polynomial $(X-j(\tau_1))(X-j(\tau_2))$ must be monic and has integer coefficients. By a numerical calculation, we find that this polynomial should be
	\[X^2+654403829760X+5209253090426880.\]
	Immediately, we can confirm \eqref{j3-37} since $j(\tau_1)$ is a root of this polynomial.
\end{proof}

Finally, let us handle the cases when $h(\tau)=3$.

\begin{proof}[Proofs of \eqref{iden6} and \eqref{iden7}]
	This time, the CM points $[9,\pm6,4]$ and $[9,\pm3,7]$ with class numbers $h(-108)=h(-243)=3$ should be employed. We have
	\begin{equation*}
		t([9,\pm6,4]) = t\biggl(\frac{\mp1+i\sqrt{3}}{3}\biggr)=6+3\sqrt[3]{2}-9\sqrt[3]{4}\pm 3i\sqrt{3}(\sqrt[3]{2}+3\sqrt[3]{4}),
	\end{equation*}
	because $j([9,\pm6,4])=\alpha_1\pm i\beta_1$ are the complex roots of
	\[X^3\!-\!151013228706000X^2\!+\!224179462188000000X\!-\!2^{12}\!\cdot\!3^3\!\cdot\!5^9\!\cdot\!11^6\!\cdot\!17^3,\]
	where
	\[\alpha_1=6000(8389623817-3329424418\sqrt[3]{2}-2642565912\sqrt[3]{4}),\]
	\[\beta_1=-12000\sqrt{3}(1664712209\sqrt[3]{2} -1321282956\sqrt[3]{4}).\]
	And we have
	\begin{equation*}
		t([9,\pm3,7]) = t\biggl(\frac{\mp1+3i\sqrt{3}}{6}\biggr)=96-28\sqrt[3]{3}+36\sqrt[3]{9}\pm 4i\sqrt{3}(7\sqrt[3]{3}+9\sqrt[3]{9}),
	\end{equation*}
	because $j([9,\pm3,7])=\alpha_2\pm i\beta_2$ are the complex roots of
	\[X^3\!+\!1855762905734664192000X^2\!-\!2^{30}\!\cdot\!3^3\!\cdot\!5^6\!\cdot\!7\!\cdot\!29\!\cdot\!1097\!\cdot\!37181X\!+\!2^{45}\!\cdot\!3\!\cdot\!5^9\!\cdot\!11^3\!\cdot\!23^3,\]
	where
	\[\alpha_2=-4096000\bigl(151022371885959-52356532113152\sqrt[3]{3}-36301991826555\sqrt[3]{9}\bigr),\]
	\[\beta_2=69632000\sqrt[6]{3}\bigl(6406233851745-3079796006656\sqrt[3]{9}\bigr).\]
	The above equalities can be verified using the same fact as used in the proofs of \eqref{iden6} and \eqref{iden7}. The difference is that algebraic numbers involved here are cubic.
\end{proof}

\section{Linking to the \texorpdfstring{$L$}{L}-values of elliptic curves}\label{LTTLOIEC}

The goal of this section is to prove Theorem \ref{detmahler} and Theorem \ref{modifiedMMasEllipticcurve}. Our process for the former is directed by Beilinson's conjecture. Readers interested in a detailed formulation of Beilinson's conjecture for curves over number fields can consult \cite{DJZ06}.

First, for simplicity, we denote the elliptic curve
\begin{equation*}
	\begin{split}
		C_{\sqrt[3]{729+405\sqrt{3}}}:Y^2 & = X^3-1062882(26+15\sqrt{3}) X^2+18596183472 (23859+13775 \sqrt{3}) X \\
		&\quad -488038239039168(3650401+2107560\sqrt{3})
	\end{split}
\end{equation*}
as $C$ directly. It is defined over $K=\Q(\sqrt{3})$, and the nontrivial element $\sigma\in\Gal(K/\Q)$ converts $C$ to the curve $C^\sigma:=C_{\sqrt[3]{729-405\sqrt{3}}}$. Since $C$ has $j$-invariant $1417905000-818626500\sqrt{3}$, conductor norm $36$ and torsion structure $\Z/2\Z$. One can search in LMFDB \cite{LMFDB} and find that $C$ is isomorphic over $K$ to the elliptic curve with LMFDB label \href{https://www.lmfdb.org/EllipticCurve/2.2.12.1/36.1/a/3}{\texttt{2\!.\!2\!.\!12\!.\!1-36\!.\!1-a3}}. Its $L$-function $L(C,s)$ with label \href{https://www.lmfdb.org/L/4/72e2/1.1/c1e2/0/2}{\texttt{4-72e2-1\!.\!1-c1e2-0-2}} can be written as
\[L(C,s)=L(f_{36},s)L(f_{144},s),\]
where
\begin{equation*}
	\begin{split}
		f_{36}(\tau)&=\eta(6\tau)^4=q-4q^7+2q^{13}+8q^{19}-5q^{25}-4q^{31}-10q^{37}+8q^{43}+9q^{49}+14q^{61}+\cdots, \\
		f_{144}(\tau)&=\frac{\eta(12\tau)^{12}}{\eta(6\tau)^4\eta(24\tau)^4}=q+4q^7+2q^{13}-8q^{19}-5q^{25}+4q^{31}-10q^{37}-8q^{43}+9q^{49}+\cdots.
	\end{split}
\end{equation*}
Moreover, $C$ is a $\Q$-curve, i.e., it is isogenous over $\overline{K}$ to $C^\sigma$, the only Galois conjugate of $C$. In fact, there is an isogeny $\phi:C\to C^\sigma$ defined over $K$ with kernel
{\small
\[\left\{O,\bigl(39366 (362+209\sqrt{3}),0\bigr),\biggl(78732(265+153\sqrt{3}),\pm38263752 \sqrt{3388314+1956244\sqrt{3}}\biggr)\right\}.\]
}

\noindent Thus, we have $L(C^\sigma,s)=L(C,s)$ since two isogenous elliptic curves over a number field have the same $L$-function \cite{Fal83}.
By using Vélu's formula \cite[Theorem 12.16]{Was08}, one can write out this isogeny explicitly:
\[\phi:(X,Y)\mapsto\left(\frac{X\phi_3(X)}{4\phi_1(X)\phi_2(X)^2},\frac{-Y\phi_4(X)}{8\phi_1(X)^2\phi_2(X)^3}\right),\]
where
{\small
\[\phi_1(X)=X-39366\bigl(362+209\sqrt{3}\bigr),\quad\phi_2(X)=X-78732\bigl(265+153\sqrt{3}\bigr),\]
\begin{equation*}
	\begin{split}
		\phi_3(X) & = \bigl(1351-780\sqrt{3}\bigr)X^3+629856\bigl(45-26\sqrt{3}\bigr)X^2+86782189536\bigl(3+\sqrt{3}\bigr)X \\
		& \quad -3904305912313344\bigl(362+209\sqrt{3}\bigr),
	\end{split}
\end{equation*}
\begin{equation*}
	\begin{split}
		\phi_4(X) & = \bigl(70226\!-\!40545\sqrt{3}\bigr)X^5\!+\!78732\bigl(698\!-\!403\sqrt{3}\bigr)X^4\!-\!24794911296\bigl(82\!-\!49\sqrt{3}\bigr)X^3\\
		& \quad -1952152956156672\bigl(425\!+\!246\sqrt{3}\bigr)X^2\!+\!153696906544127099904\bigl(56089\!+\!32383 \sqrt{3}\bigr)X\\
		& \quad -12100864846032214829641728\bigl(2672279+1542841\sqrt{3}\bigr).
	\end{split}
\end{equation*}
}

\noindent We can also obtain an isogeny $\phi^\sigma:C^\sigma\to C$ by applying $\sigma$ to the coefficients of $\phi$. Some (complicated) algebraic calculations imply that
\[\phi\circ\phi^\sigma=[4],\quad\phi^*\omega_{C^\sigma}=-(52+30\sqrt{3})\omega_C,\]
where $\omega_C$ and $\omega_{C^\sigma}$ are the invariant differentials of $C$ and $C^\sigma$ defined by $\frac{dX}{2Y}$, respectively.

Recall that for a pair of meromorphic functions $f,g$ on the Riemann surface $C(\C)$ or $C^\sigma(\C)$, there is a classical differential form
\[\eta(f,g)=\log|f|d\arg g-\log|g|d\arg f.\]
To construct regulator integrals that relate to $m_3(729\pm405\sqrt{3})$ and match them with Beilinson's conjecture, we need to find a pair of rational functions $f,g$ on $C$ that are defined over $K$ (this is for some $K$-theoretical considerations). Since the inverse of \eqref{rationaltransformation} is
\[x=\frac{-18k^2X}{3k^3X+Y-36k^9+972k^6},\quad y=1-\frac{2Y}{3k^3X+Y-36k^9+972k^6},\]
the functions $f=x^3,g=y$ on $C$ are defined over $K$. From now on, we fix this pair of $f,g$.

Next, we construct the integral paths of $\eta(f,g)$ and $\eta(f^\sigma,g^\sigma)$. Let
\[u^\pm(x)=\sqrt[3]{-1-x^3+\sqrt{x^6-(106\pm60\sqrt{3})x^3+1}},\]
where $\sqrt{\cdot}$ and $\sqrt[3]{\cdot}$ are chosen to be the principal branches, i.e.,
\[\sqrt{re^{i\theta}}=\sqrt{r}e^{\frac{i\theta}{2}},\quad\sqrt[3]{re^{i\theta}}=\sqrt[3]{r}e^{\frac{i\theta}{3}},\quad \text{for }r\in\R^{\geqslant 0},\theta\in(-\pi,\pi].\]
We can factorize $x^3+y^3+1-\sqrt[3]{729\pm405\sqrt{3}}xy$ as $(y-y_1^{\pm}(x))(y-y_2^{\pm}(x))(y-y_3^{\pm}(x))$, where
\begin{equation*}
	\begin{split}
		y_1^\pm(x)&=\frac{(3\pm\sqrt{3})x}{u^\pm(x)}+\frac{u^{\pm}(x)}{\sqrt[3]{2}}, \\
		y_2^\pm(x)&=e^{\frac{4\pi i}{3}}\frac{(3\pm\sqrt{3})x}{u^\pm(x)}+e^{\frac{2\pi i}{3}}\frac{u^{\pm}(x)}{\sqrt[3]{2}},\\
		y_3^\pm(x)&=e^{\frac{2\pi i}{3}}\frac{(3\pm\sqrt{3})x}{u^\pm(x)}+e^{\frac{4\pi i}{3}}\frac{u^{\pm}(x)}{\sqrt[3]{2}}.
	\end{split}
\end{equation*}
One can check that
\begin{equation*}
	\begin{split}
		\{\theta\in[-\pi,\pi]\mid |y_1^{\pm}(e^{i\theta})|\geqslant 1\}& =(-2\pi/3,2\pi/3],\\
		\{\theta\in[0,2\pi]\mid |y_2^{\pm}(e^{i\theta})|\geqslant 1\}& =(2\pi/3,2\pi],\\
		\{\theta\in[0,2\pi]\mid |y_3^{\pm}(e^{i\theta})|\geqslant 1\}& =(0,4\pi/3].
	\end{split}
\end{equation*}
Let
\begin{equation*}
	\begin{split}
		\gamma_1^\pm & =\{(e^{i\theta},y_1^\pm(e^{i\theta}))\mid -2\pi/3< \theta\leqslant 2\pi/3\},\\
		\gamma_2^\pm & =\{(e^{i\theta},y_2^\pm(e^{i\theta}))\mid 2\pi/3< \theta\leqslant 2\pi\},\\
		\gamma_3^\pm & =\{(e^{i\theta},y_3^\pm(e^{i\theta}))\mid 0< \theta\leqslant 4\pi/3\}.
	\end{split}
\end{equation*}
Some calculations show that the boundary points of $\gamma_1^\pm,\gamma_2^\pm$ and $\gamma_3^\pm$ are
\[(e^{-\frac{2\pi i}{3}},r^\pm e^{-\frac{\pi i}{3}}),\quad (e^{\frac{2\pi i}{3}},r^\pm e^{\frac{\pi i}{3}}),\quad (1,-r^\pm),\]
where $r^+=\sqrt[3]{20+12\sqrt{3}}, r^-=\sqrt[3]{6\sqrt{3}+\frac{9}{2}(\sqrt{6}-\sqrt{2})-7}$. Thus $\gamma^\pm=\gamma_1^\pm\cup\gamma_2^\pm\cup\gamma_3^\pm$ are continuous closed paths on the Riemann surfaces
\[\left\{(x,y)\in\C^2\ \Big{|}\ x^3+y^3+1-\sqrt[3]{729\pm405\sqrt{3}}xy=0\right\}.\]
The paths of $y_1^\pm(e^{i\theta}),y_2^\pm(e^{i\theta})$ and $y_3^\pm(e^{i\theta})$ are shown in Figure \ref{pathofy_pic} to illustrate this. We define the positive orientations of $\gamma^\pm$ in terms of $e^{i\theta}$ running counterclockwise and regard them also as closed paths on $C(\C)$ and $C^\sigma(\C)$ via \eqref{rationaltransformation}. Since complex conjugation reverses the orientations of $\gamma^{\pm}$, we have $\gamma^+\in H_1(C(\C),\Z)^-$ and $\gamma^-\in H_1(C^\sigma(\C),\Z)^-$.

\begin{figure}[h]
	\centering
	\subfigure{
		\begin{tikzpicture}[scale=0.72]
			\draw[very thin,->] (-4,0) -- (4,0);
			\draw[very thin,->] (0,-4) -- (0,4);
			\draw[dashed] (0,0) circle (1);
			\draw (0:3.26) arc (0:52:3.673);
			\draw (0:3.26) arc (0:-52:3.673);
			\draw[-{Stealth[length=2mm,width=1.2mm]}] (0:3.26) arc (0:20:3.673);
			\draw (120:3.26) arc (120:172:3.673);
			\draw (120:3.26) arc (120:68:3.673);
			\draw[-{Stealth[length=2mm,width=1.2mm]}] (-2.14,2.46) arc (125:126:10);
			\draw (240:3.26) arc (240:292:3.673);
			\draw (240:3.26) arc(240:188:3.673);
			\draw[-{Stealth[length=2mm,width=1.2mm]}] (-2.34,-2.3) arc (230:231:10);
			\draw (-60:3.44) arc (-60:-52:1.35);
			\draw (-60:3.44) arc (-60:-68:1.35);
			\draw (60:3.44) arc (60:68:1.35);
			\draw (60:3.44) arc (60:52:1.35);
			\draw (180:3.44) arc (180:188:1.35);
			\draw (180:3.44) arc (180:172:1.35);
			\fill (60:3.44) circle (2pt);
			\fill (-60:3.44) circle (2pt);
			\fill (180:3.44) circle (2pt);
			\node[right] at (20:3.35) {{\footnotesize$y_1^{+}$}};
			\node[below right] at (130:3.5) {{\footnotesize$y_2^{+}$}};
			\node[above right] at (225:3.45) {{\footnotesize$y_3^{+}$}};
		\end{tikzpicture}
	}\hspace{1.7cm}
	\subfigure{
		\begin{tikzpicture}[scale=1.152]
			\draw[very thin,->] (-2.5,0) -- (2.5,0);
			\draw[very thin,->] (0,-2.5) -- (0,2.5);
			\draw[dashed] (0,0) circle (1);
			\draw (60:2.004) arc (60:-51:1.02);
			\draw (60:2.004) arc (60:171:1.02);
			\draw (-60:2.004) arc (-60:-171:1.02);
			\draw (-60:2.004) arc (-60:51:1.02);
			\draw (180:2.004) arc (180:69:1.02);
			\draw (180:2.004) arc (180:291:1.02);
			\draw[-{Stealth[length=2mm,width=1.2mm]}] (-1.6,-0.815) arc (237:238:10);
			\draw (0:1.1) arc (180:120:0.07);
			\draw (0:1.1) arc (180:240:0.07);
			\draw (120:1.1) arc (-60:0:0.07);
			\draw (120:1.1) arc (-60:-120:0.07);
			\draw (240:1.1) arc (60:0:0.07);
			\draw (240:1.1) arc (60:120:0.07);
			\draw[-{Stealth[length=2mm,width=1.2mm]}] (-1.45,0.905) arc (120:121:10);
			\draw[-{Stealth[length=2mm,width=1.2mm]}] (1.48,0.6) arc (-9:-8:10);
			\fill (60:2.004) circle (1.3pt);
			\fill (-60:2.004) circle (1.3pt);
			\fill (180:2.004) circle (1.3pt);
			\node[right] at (25:1.67) {{\footnotesize$y_1^-$}};
			\node[below right] at (145:2.45) {{\footnotesize$y_2^-$}};
			\node[above right] at (215:2.45) {{\footnotesize$y_3^-$}};
		\end{tikzpicture}
	}
	\caption{The paths of $y^\pm_1(e^{i\theta}),y^\pm_2(e^{i\theta})$ and $y^\pm_3(e^{i\theta})$ on $\C$.}\label{pathofy_pic}
\end{figure}
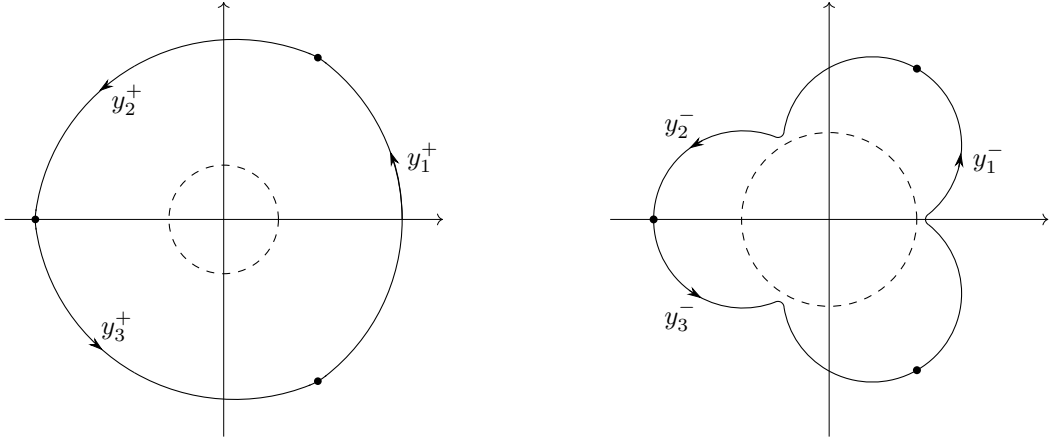

After completing all these preparations, Beilinson's conjecture then predicts that the regulator
\begin{equation}\label{theregulator}
	\left|\det{\begin{pmatrix}\frac{1}{2\pi}\int_{\gamma^+}\eta(f,g) & \frac{1}{2\pi}\int_{\gamma^-}\eta(f^\sigma,g^\sigma)\\
		\frac{1}{2\pi}\int_{\gamma^+}\eta(\phi^*f^\sigma,\phi^*g^\sigma) & \frac{1}{2\pi}\int_{\gamma^-}\eta((\phi^*f^\sigma)^\sigma,(\phi^*g^\sigma)^\sigma)\end{pmatrix}}\right|
\end{equation}
should be some rational multiple of $\frac{1}{\pi^4}L(C,2)$. By Jensen's formula, we can calculate that
\begin{equation*}
	\begin{split}
		\frac{1}{2\pi}\int_{\gamma^+}\eta(f,g) & = \frac{1}{2\pi}\int_{\gamma^+}\eta(x^3,y)\\
		& = \frac{1}{2\pi}\int_{\gamma^+}\log|x^3|\im\left(\frac{dy}{y}\right)-\log|y|\im\left(\frac{dx^3}{x^3}\right)\\
		& = -\frac{3}{2\pi}\int_{\gamma^+}\log|y|\im\left(\frac{dx}{x}\right)\\
		& =-\frac{3}{2\pi}\biggl(\int_{-\frac{2\pi}{3}}^\frac{2\pi}{3}\log|y_1^+(e^{i\theta})|d\theta+\int_{\frac{2\pi}{3}}^{2\pi}\log|y_2^+(e^{i\theta})|d\theta+\int_{0}^\frac{4\pi}{3}\log|y_3^+(e^{i\theta})|d\theta\biggr)\\
		& = -m_3(729+405\sqrt{3}).
	\end{split}
\end{equation*}
Similarly, we have $\frac{1}{2\pi}\int_{\gamma^-}\eta(f^\sigma,g^\sigma) =-m_3(729-405\sqrt{3})$ and
\begin{equation*}
	\begin{split}
		\frac{1}{2\pi}\int_{\gamma^+}\eta(\phi^*f^\sigma,\phi^*g^\sigma) & =\frac{1}{2\pi}\int_{\phi_*\gamma^+}\eta(f^\sigma,g^\sigma),\\
		\frac{1}{2\pi}\int_{\gamma^-}\eta((\phi^*f^\sigma)^\sigma,(\phi^*g^\sigma)^\sigma) & =\frac{1}{2\pi}\int_{\gamma^-}\eta((\phi^\sigma)^*f,(\phi^\sigma)^*g)\\
		& = \frac{1}{2\pi}\int_{(\phi^\sigma)_*\gamma^-}\eta(f,g).
	\end{split}
\end{equation*}
Thus, in order to work out the second row of \eqref{theregulator}, we should determine the pushforward of the paths $\gamma^+$ and $\gamma^-$ by $\phi$ and $\phi^\sigma$.

\begin{lem}
	We have $\phi_*\gamma^+=-\gamma^-\in H_1(C^\sigma(\C),\Z)^-$ and $(\phi^\sigma)_*\gamma^-=-4\gamma^+\in H_1(C(\C),\Z)^-$.
\end{lem}
\begin{proof}
	Again, we can prove this by ``numerical method''. By \eqref{rationaltransformation}, the invariant differential of $C_k$ can be written as
	\[\frac{dX}{2Y}=\frac{1+y-\frac{x(3 x^2-ky)}{k x-3 y^2}}{6k^2(y-1)(kx+3y+3)}dx.\]
	This enables us to calculate numerically (by using \textsf{Mathematica}) the integrations of $\omega_C$ and $\omega_{C^\sigma}$ along $\gamma^+$ and $\gamma^-$:
	\begin{equation}\label{approx1}
		\int_{\gamma^+}\omega_C\approx 0.000735163130i,\quad \int_{\gamma^-}\omega_{C^\sigma}\approx0.076428679590i.
	\end{equation}
	Comparing these values with the real and complex periods of the lattices corresponding to $C$ and $C^\sigma$ calculated by \textsf{SageMath}, we find that $\gamma^+$ and $\gamma^-$ are in fact generators of $H_1(C(\C),\Z)^-$ and $H_1(C^\sigma(\C),\Z)^-$, respectively. Thus, there must exist integers $a,b\in\Z$ such that $\phi_*\gamma^+=a\gamma^-,(\phi^\sigma)_*\gamma^-=b\gamma^+$. Since $\phi\circ\phi^\sigma=[4],$ we also have $ab=4$. Moreover, one can calculate that
	\begin{equation}\label{approx2}
		\int_{\phi_*\gamma^+}\omega_{C^\sigma}=\int_{\gamma^+}\phi^*\omega_{C^\sigma}=-(52+30\sqrt{3})\int_{\gamma^+}\omega_C\approx-0.076428679590i.
	\end{equation}
	Comparing \eqref{approx1} and \eqref{approx2}, we immediately observe that $a=-1$ and thus $b=-4$.
\end{proof}

According to the above lemma, the regulator \eqref{theregulator} equals
\[\left|\det{\begin{pmatrix}m_3(729+405\sqrt{3})& m_3(729-405\sqrt{3})\\
	m_3(729-405\sqrt{3})&4m_3(729+405\sqrt{3})\end{pmatrix}}\right|\]

\begin{proof}[Proof of Theorem \ref{detmahler}]
	Since the Sturm bound for $\mathcal{M}_2(\Gamma_0(144))$ is $48$, we can prove that
	\[F_{144}(\tau)=\frac{1}{2}f_{36}(\tau)-2f_{36}(4\tau)+\frac{1}{2}f_{144}(\tau),\quad \widetilde{F}_{144}(\tau)=-\frac{1}{4}f_{36}(\tau)+f_{36}(4\tau)+\frac{1}{4}f_{144}(\tau).\]
	And thus
	\[L(F_{144},2)=\frac{3}{8}L(f_{36},2)+\frac{1}{2}L(f_{144},2),\quad L(\widetilde{F}_{144},2)=-\frac{3}{16}L(f_{36},2)+\frac{1}{4}L(f_{144},2).\]
	By \eqref{iden3}, we have
	\begin{equation*}
		\begin{split}
			& 4m_3(729+405\sqrt{3})^2-m_3(729-405\sqrt{3})^2 \\
			=\ & 4\left(\frac{81}{\pi^2}\left(\frac{3}{8}L(f_{36},2)+\frac{1}{2}L(f_{144},2)\right)\right)^2-\left(\frac{324}{\pi^2}\left(-\frac{3}{16}L(f_{36},2)+\frac{1}{4}L(f_{144},2)\right)\right)^2\\
			=\ & \frac{19683}{\pi^4}L(f_{36},2)L(f_{144},2)\\
			=\ & \frac{19683}{\pi^4}L(C,2)\\
			=\ & \frac{243}{8}L''(C,0),
		\end{split}
	\end{equation*}
	where the last equality follows by the  functional equation of $L(C,s)$.
\end{proof}

Finally, let us prove Theorem \ref{modifiedMMasEllipticcurve}. To achieve this, we need the following hypergeometric formula for $\tilde{n}(k)$ proved by Samart.

\begin{thm}[{\cite[Theorem 1]{Sam23}}]\label{Samhypergeometricformula}
	Let $\tilde{n}(k)$ be the modified Mahler measure \eqref{modifiedMM}. Then for $k\in(-1,3)-\{0\}$, the following identity is true:
	\[\tilde{n}(k)=\frac{4}{1-3\sgn(k)}\re\biggl(\log k-\frac{2}{k^3}{}_4 F_3\biggl(\begin{matrix}\frac{4}{3},\ \frac{5}{3},\ 1,\ 1\\[2pt]2,\ 2,\ 2\end{matrix}\ \bigg{|}\ \frac{27}{k^3}\biggr)\biggr).\]
\end{thm}
For $k\in\C-\mathcal{K}_Q^\circ$, we also have \cite[Theorem 3.1]{Rog11}
\begin{equation}\label{Rogformula}
	m(Q_k)=\re\biggl(\log k-\frac{2}{k^3}{}_4 F_3\biggl(\begin{matrix}\frac{4}{3},\ \frac{5}{3},\ 1,\ 1\\[2pt]2,\ 2,\ 2\end{matrix}\ \bigg{|}\ \frac{27}{k^3}\biggr)\biggr).
\end{equation}

\begin{proof}[Proof of Theorem \ref{modifiedMMasEllipticcurve}]
	Since $t(\tau)$ is a Hauptmodul for $\Gamma_0(3)$, the map $\tau\mapsto t(\tau)$ is a biholomorphic mapping form the genus zero Riemann surface $X_0(3)$ to $\mathbb{P}^1(\C)$. Also, it is known from \cite[\S 14]{RV98} that $t\bigl(\pm\frac{1}{2}+\frac{i}{2\sqrt{3}}\bigr)=0.$ Thus, we have
	\[t(\tau)\neq 0,\quad\forall\tau\in\mathcal{F}'-\left\{\pm\frac{1}{2}+\frac{i}{2\sqrt{3}}\right\}.\]
	By Rodriguez Villegas' formula \eqref{Villegasformula} and the above hypergeometric formula \eqref{Rogformula}, the equation
	\[\re\biggl(\frac{\log t(\tau)}{3}\!-\!\frac{2}{t(\tau)}{}_4 F_3\biggl(\begin{matrix}\frac{4}{3},\ \frac{5}{3},\ 1,\ 1\\[2pt]2,\ 2,\ 2\end{matrix}\ \bigg{|}\ \frac{27}{t(\tau)}\biggr)\biggr)=\frac{27\sqrt{3}\im(\tau)}{4\pi^2}\underset{m,n\in\Z}{{\sum}'}\frac{\chi_{-3}(n)(3m\re(\tau)\!+\!n)}{\left|3m\tau\!+\!n\right|^4}\]
	holds on some open set of $\mathcal{F}'$. Since both sides of the above equation are harmonic on $\mathcal{F}'^\circ$, they must coincide for every $\tau\in\mathcal{F}'-\bigl\{\pm\frac{1}{2}+\frac{i}{2\sqrt{3}}\bigr\}$. Hence, according to Theorem \ref{Samhypergeometricformula}, we just proved that for $\tau\in\mathcal{F}'$ with $\sqrt[3]{t(\tau)}\in (-1,3)-\{0\}$, the following formula for $\tilde{n}(k)$ holds:
	\begin{equation}\label{latticesumformodifiedMM}
		\tilde{n}(\sqrt[3]{t(\tau)})=\frac{27\sqrt{3}\im(\tau)}{\bigl(1-3\sgn(\sqrt[3]{t(\tau)})\bigr)\pi^2}\underset{m,n\in\Z}{{\sum}'}\frac{\chi_{-3}(n)(3m\re(\tau)+n)}{\left|3m\tau+n\right|^4}.
	\end{equation}

	Now, we are back to our familiar track. From the outputs of Algorithm \ref{Searchcmpoints}, we find that $t([9,3,1])=t\bigl(-\frac{1}{6}+\frac{i}{2\sqrt{3}}\bigr)=24$. Note that this point is not listed in Table \ref{CMpointsandttau}, because we only listed in Table \ref{CMpointsandttau} those points that make $\sqrt[3]{t(\tau)}\in\C-\mathcal{K}_Q^\circ$. By \eqref{latticesumformodifiedMM}, we immediately obtain that
	\[\tilde{n}(\sqrt[3]{24})=-\frac{81}{4\pi^2}L(f_{27},2)=-3L'(f_{27},0),\]
	where $f_{27}(\tau)=\frac{1}{6}\underset{m,n\in\Z}{{\sum}}\chi_{-3}(n)(m+2n)q^{m^2+mn+n^2}=\eta(3\tau)^2\eta(9\tau)^2$ is the unique normalized cusp form in $\mathcal{M}_2(\Gamma_0(27))$. This is exactly the newform that corresponds to $C_{\sqrt[3]{24}}$ which is isomorphic over $\Q$ to the elliptic curve with LMFDB label \href{https://www.lmfdb.org/EllipticCurve/Q/27/a/1}{\texttt{27\!.\!a1}}.
\end{proof}

\section{Some remarks}

Our approach was originally inspired by the work \cite{HSY18} of Huber, Schultz and Ye on $1/\pi$-series. In fact, Q. He and Ye have already proved in \cite{HY22} all formulas conjectured by Samart in \cite{Sam14} that involve the Mahler measure of the trivariate Laurent polynomial
\[\biggl(x+\frac{1}{x}\biggr)^2\biggl(y+\frac{1}{y}\biggr)^2\frac{(1+z)^3}{z^2}-s.\]
They also suggested that Samart's conjectural identities associated to
\[\biggl(x+\frac{1}{x}\biggr)\biggl(y+\frac{1}{y}\biggr)\biggl(z+\frac{1}{z}\biggr)+\sqrt{s}\quad\text{and}\quad x^4+y^4+z^4+1-\sqrt[4]{s}xyz\]
might be proved using their method. Moreover, in \cite{Fei21}, Fei expressed the Mahler measures of $23$ families of Laurent polynomials in terms of Kronecker-Eisenstein series. There seems to be a huge number of identities that can be proved. Finally, it will be interesting if one could prove an identity that relates a $3\times 3$ or higher order determinant with Mahler measures as entries to the $L$-value of an elliptic curve.

\section*{Acknowledgement}

We would like to thank Detchat Samart for the valuable discussions about his modified Mahler measure $\tilde{n}(k)$. We also express our thanks to Dr. Yu Xu for his helpful comments. Finally, we are grateful to the referees for their comments and suggestions to improve this paper.

\end{document}